\documentclass[a4, 12pt]{amsart}
\usepackage{amssymb}
\usepackage{amstext}
\usepackage{amsmath}
\usepackage{amscd}
\usepackage{amsthm}
\usepackage{amsfonts}
\usepackage{enumerate}
\usepackage{graphicx}
\usepackage{latexsym}
\usepackage[all]{xy}
\theoremstyle{plain}
\newtheorem{thm}{Theorem}[section]

\newtheorem{lem}[thm]{Lemma}

\newtheorem{claim}{Claim}
\theoremstyle{definition}
\newtheorem{dfn}[thm]{Definition}
\newtheorem{rem}[thm]{Remark}

\newtheorem*{conv}{Convention}

\theoremstyle{remark}
\newtheorem*{ac}{Acknowledgments}
\newtheorem*{cpf}{{\it Proof of Claim}}

\numberwithin{equation}{thm}
\renewcommand{\qedsymbol}{$\blacksquare$}

\def\Ext{\operatorname{Ext}}
\def\Tor{\operatorname{Tor}}
\def\mod{\operatorname{\mathsf{mod}}}
\def\DD{\operatorname{\mathsf{D^b}}}
\def\P{\operatorname{\mathsf{perf}}}
\def\thick{\operatorname{\mathsf{thick}}}
\def\rest{\operatorname{\mathsf{rest}}}
\def\CM{\operatorname{\mathsf{CM}}}
\def\lCM{\operatorname{\underline{\mathsf{CM}}}}
\def\Spec{\operatorname{\mathsf{Spec}}}
\def\Sing{\operatorname{\mathsf{Sing}}}
\def\nat{\operatorname{\mathsf{can}}}
\def\sus{\mathsf{\Sigma}}
\def\m{\mathfrak m}
\def\p{\mathfrak p}
\def\q{\mathfrak q}
\def\A{{\mathbf A}}
\def\B{{\mathbf B}}
\def\C{{\mathbf C}}
\def\D{{\mathbf D}}
\def\E{{\mathbf E}}

\def\X{{\mathcal X}}
\def\Y{{\mathcal Y}}
\def\Z{{\mathcal Z}}
\def\WW{{\mathcal W}}

\def\T{{\mathcal T}}
\def\U{{\mathcal U}}
\def\AA{{\mathcal A}}
\def\BB{{\mathcal B}}
\def\CC{{\mathcal C}}
\def\W{{\mathbb I}}
\def\V{{\mathbb N}}
\def\lSupp{{\mathbb S}}
\def\Q{{\mathbb J}}
\def\ZZ{{\mathbb Z}}
\begin{document}
\setlength{\baselineskip}{15pt}
\title[Thick subcategories over Gorenstein local rings]{Thick subcategories over Gorenstein local rings that are locally hypersurfaces on the punctured spectra}
\author{Ryo Takahashi}
\address{Department of Mathematical Sciences, Faculty of Science, Shinshu University, 3-1-1 Asahi, Matsumoto, Nagano 390-8621, Japan}
\email{takahasi@math.shinshu-u.ac.jp}
\begin{abstract}
Let $R$ be a Gorenstein local ring which is locally a hypersurface on the punctured spectrum.
In this paper, we classify thick subcategories of the bounded derived category of finitely generated $R$-modules.
Moreover, using this classification, we also classify thick subcategories of finitely generated $R$-modules, and find out the relationships with thick subcategories of Cohen-Macaulay $R$-modules.
\end{abstract}
\maketitle
\tableofcontents

\section{Introduction}

A thick subcategory of a triangulated category is by definition a full triangulated subcategory which is closed under direct summands.
The classification problem of thick subcategories of a given triangulated category has been studied in stable homotopy theory, ring theory, modular representation theory and algebraic geometry.
The first study was done by Devinatz, Hopkins and Smith \cite{DHS,HS}, who classified thick subcategories of the category of compact objects in the $p$-local stable homotopy category.
Later on, Hopkins \cite{H} and Neeman \cite{N1} classified thick subcategories of the derived category of perfect complexes over a commutative noetherian ring in terms of specialization-closed subsets of the prime ideal spectrum of the ring.
This Hopkins-Neeman theorem was generalized by Thomason \cite{T} to quasi-compact and quasi-separated schemes.
Benson, Carlson and Rickard \cite{BCR} gave a classification theorem of thick subcategories of the stable category of finitely generated representations of a finite $p$-group in terms of closed homogeneous subvarieties of the maximal ideal spectrum of the group cohomology ring.
This was extended by Friedlander and Pevtsova \cite{FP} to finite group schemes.
Many other results on classifying thick subcategories and related results have been obtained so far; see \cite{ABCIP,Ba1,Ba2,Ba3,BIK1,BIK2,BIK3,Br,BKS,I1,I2,K1,K2,hcls}.

Let $R$ be a Gorenstein local ring which is locally a hypersurface on the punctured spectrum.
Denote by $\mod(R)$ the category of finitely generated $R$-modules, by $\DD(R)$ the bounded derived category of $\mod(R)$, by $\CM(R)$ the category of (maximal) Cohen-Macaulay $R$-modules and by $\lCM(R)$ the stable category of $\CM(R)$.
Recently, as a higher dimensional version of the work of Benson, Carlson and Rickard, Takahashi \cite{stcm} gave a classification theorem of thick subcategories of $\lCM(R)$.
In this paper, we will classify thick subcategories of $\DD(R)$ by taking the {\em infinite projective dimension loci} of those subcategories, which are specialization-closed subsets of $\Spec(R)$ contained in the singular locus $\Sing(R)$.
Moreover, using this classification, we will also classify thick subcategories of $\mod(R)$, and find out the relationships among thick subcategories of $\DD(R)$, $\mod(R)$, $\CM(R)$ and $\lCM(R)$.
The main result of this paper will be stated and proved in Section 5.

\begin{conv}
Throughout the rest of this paper, let $R$ be a commutative Gorenstein local ring of Krull dimension $\dim R=d$.
Denote by $\m$ the maximal ideal of $R$ and by $k$ the residue field of $R$.
\end{conv}

\section{The main result of \cite{stcm}}

In this section, we recall the main result of the paper \cite{stcm}, which will form the basis of the main result of the present paper.
We denote the category of finitely generated $R$-modules by $\mod(R)$, the full subcategory of (maximal) Cohen-Macaulay $R$-modules by $\CM(R)$ and its stable category by $\lCM(R)$.
We can regard a full subcategory of $\CM(R)$ as a full subcategory of $\mod(R)$.

First of all, we state the definitions of thick subcategories.

\begin{dfn}
\begin{enumerate}[(1)]
\item
A subcategory $\X$ of a category $\CC$ is called {\em strict} provided that $\X$ is closed under isomorphisms: if $X$ is an object in $\X$ and $Y$ is an object isomorphic to $X$ in $\CC$, then $Y$ is also an object in $\X$.
\item
A nonempty strict full subcategory $\X$ of a triangulated category $\T$ is called {\em thick} provided that the following hold.
\begin{enumerate}[(a)]
\item
$\X$ is closed under direct summands: if $X$ is an object in $\X$ and $Y$ is a direct summand of $X$ in $\T$, then $Y$ is also an object in $\X$.
\item
$\X$ is closed under exact triangles: for an exact triangle $L \to M \to N \to$ in $\T$, if two of $L,M,N$ are in $\X$, then so is the third.
\end{enumerate}
\item
A nonempty strict full subcategory $\X$ of $\mod(R)$ is called {\em thick} provided that the following hold.
\begin{enumerate}[(a)]
\item
$\X$ is closed under direct summands: if $X$ is an object in $\X$ and $Y$ is a direct summand of $X$ in $\mod(R)$, then $Y$ is also an object in $\X$.
\item
$\X$ is closed under short exact sequences: for an exact sequence $0 \to L \to M \to N \to 0$ of finitely generated $R$-modules, if two of $L,M,N$ are in $\X$, then so is the third.
\end{enumerate}
\item
A nonempty strict full subcategory $\X$ of $\CM(R)$ is called {\em thick} provided that the following hold.
\begin{enumerate}[(a)]
\item
$\X$ is closed under direct summands: if $X$ is an object in $\X$ and $Y$ is a direct summand of $X$ in $\CM(R)$, then $Y$ is also an object in $\X$.
\item
$\X$ is closed under short exact sequences: for an exact sequence $0 \to L \to M \to N \to 0$ of Cohen-Macaulay $R$-modules, if two of $L,M,N$ are in $\X$, then so is the third.
\end{enumerate}
\end{enumerate}
\end{dfn}

Note that a thick subcategory in each sense contains the zero object.
There are a lot of examples of a thick subcategory in each sense.
For instance, for a fixed object $X$ the full subcategories determined by vanishing of $\Tor_{\gg 0}^R(X,-)$, $\Ext_R^{\gg 0}(X,-)$ and $\Ext_R^{\gg 0}(-,X)$ are thick subcategories.
In particular, the full subcategory consisting of all objects that have finite projective dimension is thick.
More generally, the objects of complexity less than or equal to some fixed nonnegative integer form a thick subcategory.
Hence the full subcategory of modules having bounded Betti numbers and the full subcategory of modules with finite complexity are thick.

Next, we state the definitions of nonfree loci and stable supports.

\begin{dfn}
\begin{enumerate}[(1)]
\item
For an object $M$ of $\CM(R)$, we denote by $\V(M)$ the {\em nonfree locus} of $M$, namely, the set of prime ideals $\p$ of $R$ such that the $R_\p$-module $M_\p$ is nonfree.
\item
For a full subcategory $\Z$ of $\CM(R)$, we denote by $\V(\Z)$ the {\em nonfree locus} of $\Z$, namely, the union of $\V(Z)$ where $Z$ runs through all objects in $\Z$.
\item
For a subset $\Phi$ of $\Spec(R)$, we denote by $\V^{-1}(\Phi)$ the full subcategory of $\CM(R)$ consisting of all objects $M$ of $\CM(R)$ such that $\V(M)$ is contained in $\Phi$.
\end{enumerate}
\end{dfn}

\begin{dfn}
\begin{enumerate}[(1)]
\item
For an object $M$ of $\lCM(R)$, we denote by $\lSupp(M)$ the {\em stable support} of $M$, namely, the set of prime ideals $\p$ of $R$ such that $M_\p$ is not isomorphic to $0$ in $\lCM(R_\p)$.
\item
For a full subcategory $\WW$ of $\lCM(R)$, we denote by $\lSupp(\WW)$ the {\em stable support} of $\WW$, namely, the union of $\lSupp(W)$ where $W$ runs through all objects in $\WW$.
\item
For a subset $\Phi$ of $\Spec(R)$, we denote by $\lSupp^{-1}(\Phi)$ the full subcategory of $\lCM(R)$ consisting of all objects $M$ of $\lCM(R)$ such that $\lSupp(M)$ is contained in $\Phi$.
\end{enumerate}
\end{dfn}

Note that $\V^{-1}(\Phi)$ and $\lSupp^{-1}(\Phi)$ are strict subcategories of $\CM(R)$ and $\lCM(R)$, respectively.

Here we recall some definitions and introduce some notation.
The ring $R$ is called an {\em abstract hypersurface} if the $\m$-adic completion of $R$ is isomorphic to $S/(f)$ for some complete regular local ring $S$ and an element $f\in S$.
The {\em singular locus} $\Sing(R)$ is defined as the set of prime ideals $\p$ of $R$ such that the local ring $R_\p$ is singular.
A subset $\Phi$ of $\Spec(R)$ is called {\em specialization-closed} if every prime ideal of $R$ containing some prime ideal in $\Phi$ belongs to $\Phi$.
The {\em punctured spectrum} of $R$ is by definition the set $\Spec(R)\setminus\{\m\}$.
For a nonnegative integer $n$, the $n$-th {\em syzygy} $\Omega^nM$ of a finitely generated $R$-module $M$ is defined to be the image of the $n$-th differential map in a minimal free resolution of $M$.

The following theorem is proved by Takahashi \cite{stcm}, which classifies thick subcategories of $\CM(R)$ and of $\lCM(R)$.

\begin{thm}\cite[Proposition 6.2 and Theorem 6.8]{stcm}\label{stcmthm}
Consider the following two cases.
\begin{enumerate}[\rm (1)]
\item
Let $R$ be an abstract hypersurface.
Set
\begin{align*}
\A & = \{\text{Specialization-closed subsets of }\Spec(R)\text{ contained in }\Sing(R)\},\\
\B & = \{\text{Thick subcategories of }\lCM(R)\},\\
\C & = \{\text{Thick subcategories of }\CM(R)\text{ containing }R\}.
\end{align*}
\item
Let $R$ be singular, and locally an abstract hypersurface on the punctured spectrum.
Set
\begin{align*}
\A & = \{\text{Nonempty specialization-closed subsets of }\Spec(R)\text{ contained in }\Sing(R)\},\\
\B & = \{\text{Thick subcategories of }\lCM(R)\text{ containing }\Omega^dk\},\\
\C & = \{\text{Thick subcategories of }\CM(R)\text{ containing }R\text{ and }\Omega^dk\}.
\end{align*}
\end{enumerate}
In each of the above two cases, one has the following commutative diagram of bijections.
$$
\xymatrix{
& & & & & & \A \ar@<0.5mm>[llllllddd]^{\lSupp^{-1}} \ar@<0.5mm>[llddd]^{\V^{-1}} \\
\\
\\
\B \ar@<0.5mm>[rrrrrruuu]^{\lSupp} \ar@<0.5mm>[rrrr]^{\nat} & & & & \C \ar@<0.5mm>[llll]^{\nat} \ar@<0.5mm>[rruuu]^{\V}
}
$$
Here, $\nat$ denotes a canonical map.
\end{thm}

The above diagram will be extended in Theorem \ref{main} to a larger commutative diagram of bijections.

\section{Thick subcategories of $\DD(R)$}

In this section, we consider classifying thick subcategories of the bounded derived category of finitely generated $R$-modules.

Let $F:\AA\to\BB$ be a functor of categories.
For a strict full subcategory $\X$ of $\BB$ we denote by $F^{-1}\X$ the full subcategory of $\AA$ consisting of all objects $A$ of $\AA$ such that $FA$ belongs to $\X$.
Note that $F^{-1}\X$ is a strict subcategory of $\AA$.

We begin with two general results on thick subcategories of triangulated categories.

\begin{lem}\label{t/u}
Let $\T$ be a triangulated category and $\U$ a thick subcategory of $\T$.
Let $F:\T\to\T/\U$ be the localization functor.
Then there is a one-to-one correspondence
$$
\begin{CD}
\left\{
\begin{matrix}
\text{Thick subcategories}\\
\text{of }\T\text{ containing }\U
\end{matrix}
\right\}
\begin{matrix}
@>{f}>>\\
@<<{g}<
\end{matrix}
\left\{
\begin{matrix}
\text{Thick subcategories}\\
\text{of }\T/\U
\end{matrix}
\right\}
\end{CD}
$$
where $f$ is given by $\X\mapsto\X/\U$ and $g$ by $\Y\mapsto F^{-1}\Y$.
\end{lem}

\begin{proof}
Let $\X$ be a thick subcategory of $\T$ containing $\U$, and let $\Y$ be a thick subcategory of $\T/\U$.
Let us show the lemma step by step.

(1) We can regard $\U$ as a thick subcategory of the triangulated category $\X$, and hence we can define the quotient category $\X/\U$.
Let $G:\X\to\X/\U$ be the localization functor.
The inclusion functor $\alpha:\X\to\T$ is uniquely extended to a triangle functor $\beta:\X/\U\to\T/\U$ such that $F\alpha=\beta G$; see \cite[Theorem 2.1.8]{N2}.
We can show that $\beta$ is fully faithful, and hence $\X/\U$ can be viewed as a full subcategory of $\T/\U$.

Let $M$ be an object of $\X/\U$ and let $\phi:M\to N$ be an isomorphism in $\T/\U$.
Then $M$ is an object of $\X$ and $N$ is an object of $\T$.
We describe the isomorphism $\phi$ as
$$
(M\overset{s}{\gets} L\overset{h}{\to} N)=(Fh)\cdot(Fs)^{-1}.
$$
The morphism $Fh=(L\overset{1}{\gets} L\overset{h}{\to} N)=\phi\cdot(Fs)$ is an isomorphism.
Taking an exact triangle $L\overset{h}{\to} N \to A \to$ in $\T$, we observe that $A$ belongs to $\U$ by the thickness of $\U$; see \cite[Proposition 2.1.35]{N2}.
There is an exact triangle $L\overset{s}{\to} M\to B \to$ in $\T$ with $B\in\U$, which implies that $L$ belongs to $\X$.
Hence $N$ also belongs to $\X$.
Thus $N$ is an object in $\X/\U$, and it follows that $\X/\U$ is a strict subcategory of $\T/\U$.
Also, we can prove that $\X/\U$ is closed under direct summands and exact triangles.
Consequently, $\X/\U$ is a thick subcategory of $\T/\U$.

(2) By \cite[Remark 2.1.10]{N2}, the subcategory $\U$ coincides with the {\em kernel} of $F$, namely, the full subcategory of $\T$ consisting of all objects $T$ of $\T$ such that $FT$ is isomorphic to $0$ in $\T/\U$.
Since $\Y$ contains $0$ and is strict, $F^{-1}\Y$ contains $\U$.
It is clear that $F^{-1}\Y$ is closed under direct summands.
Using the fact that $F$ is a triangle functor, we observe that $F^{-1}\Y$ is closed under exact triangles.
Thus, $F^{-1}\Y$ is a thick subcategory of $\T$ containing $\U$.

(3) It is easy to check that $gf(\X)=F^{-1}(\X/\U)=\X$ and $fg(\Y)=(F^{-1}\Y)/\U=\Y$ hold.
\end{proof}

Let $F:\AA\to\BB$ be a functor of categories.
For a full subcategory $\X$ of $\AA$ we denote by $F\X$ the full subcategory of $\BB$ consisting of all objects $B$ of $\BB$ such that $B\cong FX$ for some $X\in\X$.
Note that $F\X$ is a strict subcategory of $\BB$.

The proof of the following lemma is standard, and we omit it.

\begin{lem}\label{tt'}
Let $F:\T\to\T'$ be a triangle equivalence of triangulated categories.
Let $G:\T'\to\T$ be a quasi-inverse of $F$.
Then there is a one-to-one correspondence
$$
\begin{CD}
\left\{
\begin{matrix}
\text{Thick subcategories}\\
\text{of }\T
\end{matrix}
\right\}
\begin{matrix}
@>{f}>>\\
@<<{g}<
\end{matrix}
\left\{
\begin{matrix}
\text{Thick subcategories}\\
\text{of }\T'
\end{matrix}
\right\}
\end{CD}
$$
where $f$ is given by $\X\mapsto F\X$ and $g$ by $\Y\mapsto G\Y$.
\end{lem}

We denote by $\DD(R)$ the bounded derived category of $\mod(R)$, and by $\P(R)$ the full subcategory of $\DD(R)$ consisting of all {\em perfect} complexes, namely, bounded complexes of finitely generated projective (equivalently, free) $R$-modules.

\begin{rem}\label{fact}
It is well known that $\P(R)$ is the smallest thick subcategory of $\DD(R)$ containing $R$.
\end{rem}

\begin{dfn}\label{buconst}
By virtue of \cite[Theorem 4.4.1]{B}, the assignment $M\mapsto M$ makes a triangle equivalence
\begin{equation}\label{buch}
\lCM(R) \overset{\cong}{\longrightarrow} \DD(R)/\P(R).
\end{equation}
We recall the construction of a quasi-inverse $Q_R$ of this functor which is stated in \cite[(4.5)]{B}.
Let $X$ be an object of $\DD(R)/\P(R)$.
Then $X$ is a bounded complex of finitely generated $R$-modules.
Take a free resolution
$$
F=(\cdots \overset{f_{i+1}}{\longrightarrow} F_i \overset{f_i}{\longrightarrow} F_{i-1} \overset{f_{i-1}}{\longrightarrow} \cdots)
$$
of $X$.
Fix an integer $n\ge\sup X+\dim R$, where $\sup X=\sup\{\, i\in\ZZ\mid H_i(X)\ne 0\,\}$.
Let $N$ be the image of $f_n$.
The $R$-module $N$ is Cohen-Macaulay, whence there is an exact sequence
$$
\cdots \overset{g_{i+1}}{\longrightarrow} G_i \overset{g_i}{\longrightarrow} G_{i-1} \overset{g_{i-1}}{\longrightarrow} \cdots
$$
of finitely generated free $R$-modules with $G_i=F_i$ for $i\ge n$, $g_i=f_i$ for $i\ge n+1$ and the image of $g_n$ being $N$.
Then $Q_R(X)$ is defined to be the image of $g_0$.
\end{dfn}

\begin{lem}\label{QX}
For $X\in\DD(R)/\P(R)$ and $\p\in\Spec(R)$, one has an isomorphism
$$
(Q_R(X))_\p\cong Q_{R_\p}(X_\p)
$$
in $\lCM(R_\p)$.
\end{lem}

\begin{proof}
We use the notation of Definition \ref{buconst}.
We have $n\ge\sup X+\dim R\ge\sup X_\p+\dim R_\p$.
Since $F_\p$ is a free $R_\p$-resolution of $X_\p$ and $N_\p$ is a Cohen-Macaulay $R_\p$-module, the assertion follows from the construction of the functor $Q_R$ which we observed in Definition \ref{buconst}.
\end{proof}

Now we make the definitions of the infinite projective dimension loci of an object and a full subcategory of $\DD(R)$.

\begin{dfn}
\begin{enumerate}[(1)]
\item
For an object $C$ of $\DD(R)$, we denote by $\W(C)$ the set of prime ideals $\p$ of $R$ such that the $R_\p$-complex $C_\p$ does not belong to $\P(R_\p)$, namely, $C_\p$ has infinite projective dimension as an $R_\p$-complex.
We call this the {\em infinite projective dimension locus} of $C$.
\item
For a full subcategory $\X$ of $\DD(R)$, we denote by $\W(\X)$ the union of $\W(X)$ where $X$ runs through all objects in $\X$.
We call this the {\em infinite projective dimension locus} of $\X$.
\item
For a subset $\Phi$ of $\Spec(R)$, we denote by $\W^{-1}(\Phi)$ the full subcategory of $\DD(R)$ consisting of all objects $C$ of $\DD(R)$ such that $\W(C)$ is contained in $\Phi$.
\end{enumerate}
\end{dfn}

Note that $\W^{-1}(\Phi)$ is a strict subcategory of $\DD(R)$.

We give some basic properties of infinite projective dimension loci in the next two lemmas.

\begin{lem}\label{bwx}
Let $\X$ be a full subcategory of $\DD(R)$.
Then $\W(\X)$ is a specialization-closed subset of $\Spec(R)$ contained in $\Sing(R)$.
\end{lem}

\begin{proof}
Let $\p$ be a prime ideal of $R$ and $M$ a finitely generated $R$-module.
If the $R_\p$-module $M_\p$ has finite projective dimension, then so does the $R_\q$-module $M_\q$ for every prime ideal $\q$ contained in $\p$.
On the other hand, over a regular local ring every bounded complex of finitely generated modules has finite projective dimension.
The assertion follows from these.
\end{proof}

\begin{lem}\label{3pr}
\begin{enumerate}[\rm (1)]
\item
For an object $X$ of $\P(R)$, one has $\W(X)=\emptyset$.
\item
For an object $X$ of $\DD(R)$, one has $\W(X)=\W(\sus X)$, where $\sus$ denotes the suspension functor.
\item
For an exact triangle $X \to Y \to Z \to$ in $\DD(R)$, one has $\W(X)\subseteq\W(Y)\cup\W(Z)$.
\item
For objects $X,Y\in\DD(R)$, one has $\W(X\oplus Y)=\W(X)\cup\W(Y)$.
\end{enumerate}
\end{lem}

This lemma is shown straightforwardly.

The result below is a direct consequence of Lemma \ref{3pr}.

\begin{lem}\label{winv}
Let $\Phi$ be a subset of $\Spec(R)$.
Then $\W^{-1}(\Phi)$ is a thick subcategory of $\DD(R)$ containing $\P(R)$.
\end{lem}

We state here three lemmas; the first and third ones will play an important role in the proof of the main result of this section.

\begin{lem}\label{ws}
Let $\X$ be a thick subcategory of $\DD(R)$ containing $\P(R)$.
Then the equality $\W(\X)=\lSupp(Q_R(\X/\P(R)))$ holds.
\end{lem}

\begin{proof}
Let $\p$ be a prime ideal of $R$.
Then we have
\begin{align*}
& \p\in\lSupp(Q_R(\X/\P(R))) \\
\Longleftrightarrow\quad & Q_{R_\p}(X_\p)\cong(Q_R(X))_\p\not\cong 0\text{ in }\lCM(R_\p)\text{ for some }X\in\X \\
\Longleftrightarrow\quad & X_\p\not\cong 0\text{ in }\DD(R_\p)/\P(R_\p)\text{ for some }X\in\X \\
\Longleftrightarrow\quad & X_\p\notin\P(R_\p)\text{ for some }X\in\X \\
\Longleftrightarrow\quad & \p\in\W(\X).
\end{align*}
Here, Lemma \ref{QX} is applied in the second statement.
Since the functor $Q_{R_\p}$ is an equivalence of additive categories, we obtain the second equivalence.
The third equivalence follows from the fact that $\P(R_\p)$ is the kernel of the localization functor $\DD(R_\p)\to\DD(R_\p)/\P(R_\p)$; see \cite[Remark 2.1.10]{N2}.
\end{proof}

For each $M\in\mod(R)$, we define the $R$-complex
$$
\Delta M=(\cdots \to 0 \to 0 \to M \to 0 \to 0 \to \cdots)
$$
with $M$ being in degree zero.
It is well known and easily observed that the assignment $M\mapsto \Delta M$ makes a fully faithful functor $\mod(R)\to\DD(R)$.

\begin{lem}\label{phi}
For every $\p\in\Sing R$, one has $\W(\Delta(R/\p))=V(\p)$.
\end{lem}

\begin{proof}
The set $\W(\Delta(R/\p))$ consists of the prime ideals $\q$ of $R$ such that the $R_\q$-module $R_\q/\p R_\q$ has infinite projective dimension.
Hence $\W(\Delta(R/\p))$ is a subset of $V(\p)$.
Since $R_\p$ is singular, the residue field $\kappa(\p)=R_\p/\p R_\p$ has infinite projective dimension as an $R_\p$-module.
This means that $\p$ belongs to $\W(\Delta(R/\p))$.
As $\W(\Delta(R/\p))$ is specialization-closed by Lemma \ref{bwx}, it contains $V(\p)$.
Hence we have $\W(\Delta(R/\p))=V(\p)$.
\end{proof}

\begin{lem}\label{1335}
Let $\Phi$ be a specialization-closed subset of $\Spec(R)$ contained in $\Sing(R)$.
Then the equality $\Phi=\W(\W^{-1}(\Phi))$ holds.
\end{lem}

\begin{proof}
It is clear that $\Phi$ contains $\W(\W^{-1}(\Phi))$.
Let $\p\in\Phi$.
As $\Phi$ is contained in $\Sing(R)$, Lemma \ref{phi} gives an equality $\W(\Delta(R/\p))=V(\p)$.
Since $\Phi$ is specialization-closed, $V(\p)$ is contained in $\Phi$.
Hence $\Delta(R/\p)$ is a subset of $\W^{-1}(\Phi)$, and we obtain $\p\in V(\p)=\W(\Delta(R/\p))\subseteq\W(\W^{-1}(\Phi))$.
\end{proof}

Now we state and prove the main result of this section.

\begin{thm}\label{thmder}
\begin{enumerate}[\rm (1)]
\item
Let $R$ be an abstract hypersurface.
Then one has the following one-to-one correspondence:
$$
\begin{CD}
\left\{
\begin{matrix}
\text{Thick subcategories of }\DD(R)\\
\text{containing }R
\end{matrix}
\right\}
\begin{matrix}
@>{\W}>>\\
@<<{\W^{-1}}<
\end{matrix}
\left\{
\begin{matrix}
\text{Specialization-closed subsets}\\
\text{of }\Spec(R)\text{ contained in }\Sing(R)
\end{matrix}
\right\}.
\end{CD}
$$
\item
Let $R$ be singular, and be locally an abstract hypersurface on the punctured spectrum.
Then one has the following one-to-one correspondence:
$$
\begin{CD}
\left\{
\begin{matrix}
\text{Thick subcategories of }\DD(R)\\
\text{containing }R\text{ and }k
\end{matrix}
\right\}
\begin{matrix}
@>{\W}>>\\
@<<{\W^{-1}}<
\end{matrix}
\left\{
\begin{matrix}
\text{Nonempty specialization-closed subsets}\\
\text{of }\Spec(R)\text{ contained in }\Sing(R)
\end{matrix}
\right\}.
\end{CD}
$$
\end{enumerate}
\end{thm}

\begin{proof}
(1) We have the bijections
\begin{align*}
& \{\text{Thick subcategories of }\DD(R)\text{ containing }R\} \\
\overset{e}{=}\quad & \{\text{Thick subcategories of }\DD(R)\text{ containing }\P(R)\} \\
\overset{f}{\to}\quad & \{\text{Thick subcategories of }\DD(R)/\P(R)\} \\
\overset{g}{\to}\quad & \{\text{Thick subcategories of }\lCM(R)\} \\
\overset{h}{\to}\quad & \{\text{Specialization-closed subsets of }\Spec(R)\text{ contained in }\Sing(R)\},
\end{align*}
where $f$ is given by $\X\mapsto\X/\P(R)$, $g$ by $\Y\mapsto Q_R(\Y)$ and $h$ by $\Z\mapsto\lSupp(\Z)$.
The equality $e$ follows from Remark \ref{fact}.
Lemma \ref{t/u} implies that $f$ is bijective, and so is $g$ by \eqref{buch} and Lemma \ref{tt'}.
The fact that $h$ is bijective is shown by Theorem \ref{stcmthm}.
The composition of all the above bijections sends each thick subcategory $\X$ of $\DD(R)$ containing $R$ to the specialization-closed subset $\lSupp(Q_R(\X/\P(R)))$ of $\Spec(R)$ contained in $\Sing(R)$, which coincides with $\W(\X)$ by Lemma \ref{ws}.
Thus the assignment $\X\mapsto\W(\X)$ makes a bijection from the set of thick subcategories of $\DD(R)$ containing $R$ to the set of specialization-closed subsets of $\Spec(R)$ contained in $\Sing(R)$.
Lemmas \ref{winv} and \ref{1335} guarantee that the assignment $\Phi\mapsto\W^{-1}(\Phi)$ makes the inverse map.

(2) This assertion is proved similarly to (1).
Just note that a thick subcategory of $\DD(R)$ containing $R$ and $k$ is nothing but a thick subcategory of $\DD(R)$ containing $\P(R)$ and $\Omega^dk$, and that $Q_R(\Omega^dk)$ is isomorphic to $\Omega^dk$ in $\lCM(R)$.
\end{proof}

\section{Thick subcategories of $\mod(R)$ and $\CM(R)$}

In this section, we consider classifying thick subcategories of $\mod(R)$ and $\CM(R)$ by using the classification theorem of thick subcategories of $\DD(R)$ which has been obtained in the previous section.
We start by introducing the notion of an infinite projective dimension locus for modules.

\begin{dfn}
\begin{enumerate}[(1)]
\item
For an object $M$ of $\mod(R)$, we denote by $\Q(M)$ the set of prime ideals $\p$ of $R$ such that the $R_\p$-module $M_\p$ has infinite projective dimension.
We call this the {\em infinite projective dimension locus} of $M$.
\item
For a full subcategory $\Y$ of $\mod(R)$, we denote by $\Q(\Y)$ the union of $\Q(Y)$ where $Y$ runs through all objects in $\Y$.
We call this the {\em infinite projective dimension locus} of $\Y$.
\item
For a subset $\Phi$ of $\Spec(R)$, we denote by $\Q^{-1}(\Phi)$ the full subcategory of $\mod(R)$ consisting of all objects $M$ of $\mod(R)$ such that $\Q(M)$ is contained in $\Phi$.
\end{enumerate}
\end{dfn}

Note that $\Q^{-1}(\Phi)$ is a strict subcategory of $\mod(R)$.

For a full subcategory $\Y$ of $\mod(R)$, let $\Delta\Y$ denote the full subcategory of $\DD(R)$ consisting of all complexes $\Delta Y$ with $Y\in\Y$.
The infinite projective dimension loci of finitely generated modules are closely related to the infinite projective dimension loci of bounded complexes of finitely generated modules and the nonfree loci of Cohen-Macaulay modules.

\begin{lem}\label{qwv}
\begin{enumerate}[\rm (1)]
\item
\begin{enumerate}[\rm (i)]
\item
For a finitely generated $R$-module $M$, one has $\Q(M)=\W(\Delta M)$.
\item
For a Cohen-Macaulay $R$-module $M$, one has $\Q(M)=\V(M)$.
\end{enumerate}
\item
\begin{enumerate}[\rm (i)]
\item
For a full subcategory $\Y$ of $\mod(R)$, one has $\Q(\Y)=\W(\Delta\Y)$.
\item
For a full subcategory $\Z$ of $\CM(R)$, one has $\Q(\Z)=\V(\Z)$.
\end{enumerate}
\end{enumerate}
\end{lem}

\begin{proof}
(1) The first assertion is clear.
Let $M$ be a Cohen-Macaulay $R$-module and $\p$ a prime ideal of $R$.
Then $M_\p$ is a (maximal) Cohen-Macaulay $R_\p$-module, and it is seen from the Auslander-Buchsbaum formula that the $R_\p$-module $M_\p$ has infinite projective dimension if and only if it is nonfree.
The second assertion follows from this.

(2) This is shown by making use of (1).
\end{proof}

Next we define the restrictions of each subcategory of $\DD(R)$ to $\mod(R)$ and $\CM(R)$.

\begin{dfn}
\begin{enumerate}[(1)]
\item
For a strict full subcategory $\X$ of $\DD(R)$, we denote by $\rest_{\mod}\X$ the {\em restriction} of $\X$ to $\mod(R)$, i.e., the full subcategory of $\mod(R)$ consisting of all finitely generated $R$-modules $M$ such that $\Delta M$ belongs to $\X$.
\item
For a strict full subcategory $\Y$ of $\mod(R)$, we denote by $\rest_{\CM}\Y$ the {\em restriction} of $\Y$ to $\CM(R)$, i.e., the full subcategory of $\CM(R)$ consisting of all Cohen-Macaulay $R$-modules belonging to $\Y$.
\end{enumerate}
\end{dfn}

Note that $\rest_{\mod}\X$ and $\rest_{\CM}\Y$ are strict subcategories of $\mod(R)$ and $\CM(R)$, respectively.

For a full subcategory $\Y$ of $\mod(R)$, we denote by $\Omega\Y$ the full subcategory of $\mod(R)$ consisting of all modules $\Omega Y$ with $Y\in\Y$.
The next two lemmas are concerned with relationships among restrictions, infinite projective dimension loci and nonfree loci.

\begin{lem}\label{wmod}
\begin{enumerate}[\rm (1)]
\item
If $\X$ is a thick subcategory of $\DD(R)$, then $\rest_{\mod}\X$ is a thick subcategory of $\mod(R)$.
\item
For a subset $\Phi$ of $\Spec(R)$, one has $\Q^{-1}(\Phi)=\rest_{\mod}\W^{-1}(\Phi)$.
\item
For a full subcategory $\Y$ of $\mod(R)$, one has $\Q(\Y)=\Q(\Omega\Y)$.
\end{enumerate}
\end{lem}

\begin{proof}
(1) The complex $\Delta 0$ is the zero object of $\DD(R)$.
Hence $\Delta 0$ belongs to $\X$, which says that the zero module $0$ belongs to $\rest_{\mod}\X$.
For $M,N\in\mod(R)$, we have $\Delta(M\oplus N)=\Delta M\oplus \Delta N$.
Each exact sequence $0\to L\to M\to N\to 0$ in $\mod(R)$ induces an exact triangle $\Delta L\to \Delta M\to \Delta N\to$ in $\DD(R)$.
These show that $\rest_{\mod}\X$ is closed under direct summands and short exact sequences.

(2) This follows from Lemma \ref{qwv}(1).

(3) Clear.
\end{proof}

\begin{lem}\label{wcm}
\begin{enumerate}[\rm (1)]
\item
If $\Y$ is a thick subcategory of $\mod(R)$, then $\rest_{\CM}\Y$ is a thick subcategory of $\CM(R)$.
\item
For a subset $\Phi$ of $\Spec(R)$, one has $\V^{-1}(\Phi)=\rest_{\CM}\Q^{-1}(\Phi)$.
\item
For a full subcategory $\Y$ of $\mod(R)$, one has $\Q(\Y)=\V(\Omega^d\Y)$.
\item
For a thick subcategory $\Y$ of $\mod(R)$ containing $R$, one has $\Q(\Y)=\V(\rest_{\CM}\Y)$.
\end{enumerate}
\end{lem}

\begin{proof}
(1) This is straightforward.

(2) This follows from Lemma \ref{qwv}(1)(ii).

(3) Lemmas \ref{wmod}(3) and \ref{qwv}(2) imply this assertion.

(4) We have
$$
\V(\rest_{\CM}\Y) \overset{\rm (a)}{=} \Q(\rest_{\CM}\Y) \overset{\rm (b)}{\subseteq} \Q(\Y) \overset{\rm (c)}{=} \V(\Omega^d\Y) \overset{\rm (d)}{\subseteq} \V(\rest_{\CM}\Y).
$$
Indeed, the equality (a) follows from Lemma \ref{qwv}(2).
The inclusion (b) is evident.
Assertion (3) implies the equality (c).
Since $\Y$ is thick and contains $R$, the subcategory $\Omega^d\Y$ is contained in $\rest_{\CM}\Y$.
This shows the inclusion (d).
\end{proof}

Now we can obtain a classification result of thick subcategories of $\mod(R)$.

\begin{thm}\label{thmmod}
\begin{enumerate}[\rm (1)]
\item
Let $R$ be an abstract hypersurface.
Then one has the following one-to-one correspondence:
$$
\begin{CD}
\left\{
\begin{matrix}
\text{Thick subcategories of }\mod(R)\\
\text{containing }R
\end{matrix}
\right\}
\begin{matrix}
@>{\Q}>>\\
@<<{\Q^{-1}}<
\end{matrix}
\left\{
\begin{matrix}
\text{Specialization-closed subsets}\\
\text{of }\Spec(R)\text{ contained in }\Sing(R)
\end{matrix}
\right\}.
\end{CD}
$$
\item
Let $R$ be singular, and be locally an abstract hypersurface on the punctured spectrum.
Then one has the following one-to-one correspondence:
$$
\begin{CD}
\left\{
\begin{matrix}
\text{Thick subcategories of }\mod(R)\\
\text{containing }R\text{ and }k
\end{matrix}
\right\}
\begin{matrix}
@>{\Q}>>\\
@<<{\Q^{-1}}<
\end{matrix}
\left\{
\begin{matrix}
\text{Nonempty specialization-closed subsets}\\
\text{of }\Spec(R)\text{ contained in }\Sing(R)
\end{matrix}
\right\}.
\end{CD}
$$
\end{enumerate}
\end{thm}

\begin{proof}
(1) Let $\Y$ be a thick subcategory of $\mod(R)$ containing $R$, and let $\Phi$ be a specialization-closed subset of $\Spec(R)$ contained in $\Sing(R)$.
Lemmas \ref{bwx} and \ref{qwv}(2) say that $\Q(\Y)$ is a specialization-closed subset of $\Spec(R)$ contained in $\Sing(R)$.
It follows from (1), (2) in Lemma \ref{wmod} and Lemma \ref{winv} that $\Q^{-1}(\Phi)$ is a thick subcategory of $\mod(R)$ containing $R$.
We establish two claims, which complete the proof of the first assertion of the theorem.

\begin{claim}
The equality $\Phi=\Q(\Q^{-1}(\Phi))$ holds.
\end{claim}

\begin{cpf}
It is obvious that $\Phi$ contains $\Q(\Q^{-1}(\Phi))$.
If $\p$ is a prime ideal in $\Phi$, then Lemmas \ref{phi} and \ref{qwv}(1) yield $\Q(R/\p)=\W(\Delta(R/\p))=V(\p)\subseteq\Phi$.
Hence $R/\p$ belongs to $\Q^{-1}(\Phi)$, and $\p$ is in $\Q(\Q^{-1}(\Phi))$.
\renewcommand{\qedsymbol}{$\square$}
\qed
\end{cpf}

\begin{claim}
The equality $\Y=\Q^{-1}(\Q(\Y))$ holds.
\end{claim}

\begin{cpf}
Evidently, $\Y$ is contained in $\Q^{-1}(\Q(\Y))$.
Let $M$ be an $R$-module in $\Q^{-1}(\Q(\Y))$.
Then $\Q(\Omega^dM)=\Q(M)$ by Lemma \ref{wmod}(3), which is contained in $\Q(\Y)$.
Hence $\Omega^dM$ belongs to $\rest_{\CM}\Q^{-1}(\Q(\Y))$.
We have
$$
\rest_{\CM}\Q^{-1}(\Q(\Y))=\V^{-1}(\Q(\Y))=\V^{-1}(\V(\rest_{\CM}\Y))
$$
by (2) and (4) in Lemma \ref{wcm}.
Since $\rest_{\CM}\Y$ is a thick subcategory of $\CM(R)$ containing $R$ by Lemma \ref{wcm}(1), Theorem \ref{stcmthm}(1) implies that $\V^{-1}(\V(\rest_{\CM}\Y))$ coincides with $\rest_{\CM}\Y$.
Therefore $\Omega^dM$ belongs to $\Y$.
Since $\Y$ is a thick subcategory of $\mod(R)$ containing $R$, we easily see that $M$ belongs to $\Y$.
\renewcommand{\qedsymbol}{$\square$}
\qed
\end{cpf}

(2) Let $\Y$ be a thick subcategory of $\mod(R)$ containing $R$ and $k$, and let $\Phi$ be a nonempty specialization-closed subset of $\Spec(R)$ contained in $\Sing(R)$.
As $R$ is singular, we have $\Q(k)=\{\m\}$.
Hence $\Q(\Y)$ is a nonempty specialization-closed subset of $\Spec(R)$ contained in $\Sing(R)$ by Lemmas \ref{bwx} and \ref{qwv}(2).
Since $\Phi$ is nonempty and specialization-closed, it contains the maximal ideal $\m$, and hence $k$ belongs to $\Q^{-1}(\Phi)$.
Hence $\Q^{-1}(\Phi)$ is a thick subcategory of $\mod(R)$ containing $R$ and $k$ by (1),(2) in Lemma \ref{wmod} and Lemma \ref{winv}.

(The proof of) Claim 1 gives the equality $\Phi=\Q(\Q^{-1}(\Phi))$.
The equality $\Y=\Q^{-1}(\Q(\Y))$ is obtained from the proof of Claim 2 where Theorem \ref{stcmthm}(1) is replaced with Theorem \ref{stcmthm}(2).
(Note that $\rest_{\CM}\Y$ contains $\Omega^dk$ since so does $\Y$.)
\end{proof}

We recall here the definitions of the thick closures of subcategories.

\begin{dfn}
\begin{enumerate}[(1)]
\item
For a full subcategory $\X$ of $\DD(R)$, we denote by $\thick_{\DD}\X$ the {\em thick closure} of $\X$ in $\DD(R)$, that is, the smallest thick subcategory of $\DD(R)$ containing $\X$.
\item
For a full subcategory $\Y$ of $\mod(R)$, we simply write $\thick_{\DD}\Y=\thick_{\DD}\Delta\Y$.
\item
For a full subcategory $\Y$ of $\mod(R)$, we denote by $\thick_{\mod}\Y$ the {\em thick closure} of $\Y$ in $\mod(R)$, that is, the smallest thick subcategory of $\mod(R)$ containing $\Y$.
\end{enumerate}
\end{dfn}

Thick closures do not enlarge infinite projective dimension loci.

\begin{lem}\label{clos}
\begin{enumerate}[\rm (1)]
\item
Let $\X$ be a full subcategory of $\DD(R)$.
Then $\W(\X)=\W(\thick_{\DD}\X)$.
\item
Let $\Y$ be a full subcategory of $\mod(R)$.
Then $\Q(\Y)=\Q(\thick_{\mod}\Y)$.
\end{enumerate}
\end{lem}

\begin{proof}
(1) It is trivial that $\W(\X)$ is contained in $\W(\thick_{\DD}\X)$.
Lemma \ref{winv} implies that $\W^{-1}(\W(\X))$ is thick.
Since it contains $\X$, it also contains $\thick_{\DD}\X$.
Therefore $\W(\X)$ contains $\W(\thick_{\DD}\X)$.

(2) This is similarly shown to the first assertion.
Use (1),(2) in Lemma \ref{wmod} and Lemma \ref{winv}.
\end{proof}

Restrictions and thick closures make a bijection between thick subcategories of $\DD(R)$ and $\mod(R)$.

\begin{thm}\label{dermod}
\begin{enumerate}[\rm (1)]
\item
Let $R$ be an abstract hypersurface.
Then one has the following one-to-one correspondence:
$$
\begin{CD}
\left\{
\begin{matrix}
\text{Thick subcategories of }\DD(R)\\
\text{containing }R
\end{matrix}
\right\}
\begin{matrix}
@>{\rest_{\mod}}>>\\
@<<{\thick_{\DD}}<
\end{matrix}
\left\{
\begin{matrix}
\text{Thick subcategories of }\mod(R)\\
\text{containing }R
\end{matrix}
\right\}.
\end{CD}
$$
\item
Let $R$ be singular, and be locally an abstract hypersurface on the punctured spectrum.
Then one has the following one-to-one correspondence:
$$
\begin{CD}
\left\{
\begin{matrix}
\text{Thick subcategories of }\DD(R)\\
\text{containing }R\text{ and }k
\end{matrix}
\right\}
\begin{matrix}
@>{\rest_{\mod}}>>\\
@<<{\thick_{\DD}}<
\end{matrix}
\left\{
\begin{matrix}
\text{Thick subcategories of }\mod(R)\\
\text{containing }R\text{ and }k
\end{matrix}
\right\}.
\end{CD}
$$
\end{enumerate}
\end{thm}

\begin{proof}
(1) We observe from Theorems \ref{thmder}(1) and \ref{thmmod}(1) that the assignment $\X\mapsto\Q^{-1}(\W(\X))$ makes a bijection from the set of thick subcategories of $\DD(R)$ containing $R$ to the set of thick subcategories of $\mod(R)$ containing $R$ whose inverse map is given by $\Y\mapsto\W^{-1}(\Q(\Y))$.
We have
$$
\Q^{-1}(\W(\X))=\rest_{\mod}\W^{-1}(\W(\X))=\rest_{\mod}\X
$$
by Lemma \ref{wmod}(2) and Theorem \ref{thmder}(1).
Also, we have
$$
\W^{-1}(\Q(\Y))=\W^{-1}(\W(\Delta\Y))=\W^{-1}(\W(\thick_{\DD}\Delta\Y))=\thick_{\DD}\Delta\Y=\thick_{\DD}\Y
$$
by Lemmas \ref{qwv}(2), \ref{clos}(1) and Theorem \ref{thmder}(1).

(2) In the proof of the first assertion, use Theorems \ref{thmder}(2) and \ref{thmmod}(2) instead of Theorems \ref{thmder}(1) and \ref{thmmod}(1).
\end{proof}

We obtain an analogous result to Theorem \ref{dermod}.

\begin{thm}\label{modcm}
\begin{enumerate}[\rm (1)]
\item
Let $R$ be an abstract hypersurface.
Then one has the following one-to-one correspondence:
$$
\begin{CD}
\left\{
\begin{matrix}
\text{Thick subcategories of }\mod(R)\\
\text{containing }R
\end{matrix}
\right\}
\begin{matrix}
@>{\rest_{\CM}}>>\\
@<<{\thick_{\mod}}<
\end{matrix}
\left\{
\begin{matrix}
\text{Thick subcategories of }\CM(R)\\
\text{containing }R
\end{matrix}
\right\}.
\end{CD}
$$
\item
Let $R$ be singular, and be locally an abstract hypersurface on the punctured spectrum.
Then one has the following one-to-one correspondence:
$$
\begin{CD}
\left\{
\begin{matrix}
\text{Thick subcategories of }\mod(R)\\
\text{containing }R\text{ and }k
\end{matrix}
\right\}
\begin{matrix}
@>{\rest_{\CM}}>>\\
@<<{\thick_{\mod}}<
\end{matrix}
\left\{
\begin{matrix}
\text{Thick subcategories of }\CM(R)\\
\text{containing }R\text{ and }\Omega^dk
\end{matrix}
\right\}.
\end{CD}
$$
\end{enumerate}
\end{thm}

\begin{proof}
(1) It is seen from Theorems \ref{thmmod}(1) and \ref{stcmthm}(1) that we have a bijection $\Y\mapsto\V^{-1}(\Q(\Y))$ from the set of thick subcategories of $\mod(R)$ containing $R$ to the set of thick subcategories of $\CM(R)$ containing $R$ whose inverse map is given by $\Z\mapsto\Q^{-1}(\V(\Z))$.
The equalities
$$
\V^{-1}(\Q(\Y))=\rest_{\CM}\Q^{-1}(\Q(\Y))=\rest_{\CM}\Y
$$
hold by Lemma \ref{wcm}(2) and Theorem \ref{thmmod}(1), and the equalities
$$
\Q^{-1}(\V(\Z))=\Q^{-1}(\Q(\Z))=\Q^{-1}(\Q(\thick_{\mod}\Z))=\thick_{\mod}\Z
$$
hold by Lemmas \ref{qwv}(2), \ref{clos}(2) and Theorem \ref{thmmod}(1).

(2) In the proof of the first assertion, use Theorems \ref{thmmod}(2) and \ref{stcmthm}(2) instead of Theorems \ref{thmmod}(1) and \ref{stcmthm}(1).
\end{proof}

\section{The main theorem}

Combining the theorems that have been obtained in Sections 2--4 yields the following theorem, which is the main result of this paper.
We notice that the diagram in the following theorem includes the diagram in Theorem \ref{stcmthm}.

\begin{thm}\label{main}
Consider the following two cases.
\begin{enumerate}[\rm (1)]
\item
Let $R$ be an abstract hypersurface local ring.
Set
\begin{align*}
\A & = \{\text{Specialization-closed subsets of }\Spec(R)\text{ contained in }\Sing(R)\},\\
\B & = \{\text{Thick subcategories of }\lCM(R)\},\\
\C & = \{\text{Thick subcategories of }\CM(R)\text{ containing }R\},\\
\D & = \{\text{Thick subcategories of }\mod(R)\text{ containing }R\},\\
\E & = \{\text{Thick subcategories of }\DD(R)\text{ containing }R\}.
\end{align*}
\item
Let $R$ be a $d$-dimensional Gorenstein singular local ring with residue field $k$ which is locally an abstract hypersurface on the punctured spectrum.
Set
\begin{align*}
\A & = \{\text{Nonempty specialization-closed subsets of }\Spec(R)\text{ contained in }\Sing(R)\},\\
\B & = \{\text{Thick subcategories of }\lCM(R)\text{ containing }\Omega^dk\},\\
\C & = \{\text{Thick subcategories of }\CM(R)\text{ containing }R\text{ and }\Omega^dk\},\\
\D & = \{\text{Thick subcategories of }\mod(R)\text{ containing }R\text{ and }k\},\\
\E & = \{\text{Thick subcategories of }\DD(R)\text{ containing }R\text{ and }k\}.
\end{align*}
\end{enumerate}
In each of the above two cases, one has the following commutative diagram of bijections.
$$
\xymatrix{
& & & & & & \A \ar@<0.5mm>[llllllddd]^{\lSupp^{-1}} \ar@<0.5mm>[llddd]^{\V^{-1}} \ar@<0.5mm>[rrddd]^{\Q^{-1}} \ar@<0.5mm>[rrrrrrddd]^{\W^{-1}} \\
\\
\\
\B \ar@<0.5mm>[rrrrrruuu]^{\lSupp} \ar@<0.5mm>[rrrr]^{\nat} & & & & \C \ar@<0.5mm>[llll]^{\nat} \ar@<0.5mm>[rruuu]^{\V} \ar@<0.5mm>[rrrr]^{\thick_{\mod}} & & & & \D \ar@<0.5mm>[lluuu]^{\Q} \ar@<0.5mm>[llll]^{\rest_{\CM}} \ar@<0.5mm>[rrrr]^{\thick_{\DD}} & & & & \E \ar@<0.5mm>[llll]^{\rest_{\mod}} \ar@<0.5mm>[lllllluuu]^{\W}
}
$$
\end{thm}

\begin{proof}
Combining Theorems \ref{stcmthm}, \ref{thmder}, \ref{thmmod}, \ref{dermod} and \ref{modcm}, we obtain all the bijections and the commutativity of the left triangle.
The commutativity of the right and middle triangles follows from Lemmas \ref{wmod}(2) and \ref{wcm}(2), respectively.
\end{proof}

\begin{rem}
Recently, Iyengar \cite{I2} announced that thick subcategories of the bounded derived category of finitely generated modules over a ring locally complete intersection which is essentially of finite type over a field are classified in terms of certain subsets of the prime ideal spectrum of the Hochschild cohomology ring.
It would be interesting to compare the results and approaches with ours.
\end{rem}

\begin{ac}
The author is grateful to the referee that he/she read the paper carefully and gave the author helpful comments.
\end{ac}

\end{document}